\documentclass[a4paper,11pt]{amsart}

\usepackage{amssymb,amsthm}  
\usepackage{mathtools}      
\usepackage{mathabx}        
\usepackage[bb=fourier,cal=euler,scr=rsfs]{mathalfa}	
\usepackage{enumitem}       
\usepackage[colorlinks=true]{hyperref}
\usepackage[ansinew]{inputenc}

\numberwithin{equation}{section}     

\setlist[enumerate,1]{label={\upshape(\roman*)},ref=\roman*}
\setlist[enumerate,2]{label={\upshape(\alph*)},ref=\alph*}

\def\QQ{{\mathbb Q}} \def\RR{{\mathbb R}}  \def\TT{{\mathbb T}}

\def\cA{\mathcal{A}}  \def\cG{\mathcal{G}}  
\def\cB{\mathcal{B}}    
\def\cC{\mathcal{C}}  \def\cI{\mathcal{I}} \def\cO{\mathcal{O}} \def\cU{\mathcal{U}}
\def\cD{\mathcal{D}}    
    \def\cW{\mathcal{W}}

\newtheorem*{teo*}{Theorem}

\newtheorem*{teoA}{Theorem A}
\newtheorem*{teoB}{Theorem B}

\newtheorem*{teoC}{Theorem C}
\newtheorem{teo}{Theorem}[section]

\newtheorem*{conj*}{Conjecture}
\newtheorem{quest}{Question}
\newtheorem{cor}[teo]{Corollary}
\newtheorem*{corD}{Corollary D}
\newtheorem*{af}{Claim}
\newtheorem{lema}[teo]{Lemma}
\newtheorem{prop}[teo]{Proposition}

\newcommand{\bi}{\begin{itemize}}
\newcommand{\ei}{\end{itemize}}

\theoremstyle{definition}
\newtheorem{defi}[teo]{Definition}
\theoremstyle{remark}
\newtheorem{obs}[teo]{Remark}
\newtheorem{ej}[teo]{Example}

\newcommand{\eps}{\varepsilon}

\DeclareMathOperator{\PH}{PH}

\author[S. Crovisier]{Sylvain Crovisier}
\author[R. Potrie]{Rafael Potrie}

\title[Robust transitivity versus trapping regions]{Robust transitivity \emph{versus} trapping regions\\ for partially hyperbolic diffeomorphisms}

\thanks{S.C. was partially supported by LIA-IFUMI and the ERC project 692925 \emph{NUHGD}. 
R.P. was partially supported by CSIC.}

\begin{document}
\maketitle

\begin{abstract}
We show that among partially hyperbolic diffeomorphisms with one dimensional center there is a $C^1$-open and dense subset for which either there is a proper quasi-attractor or both strong foliations are minimal. The same result holds among volume-preserving partially hyperbolic diffeomorphisms and has consequences about robust transitivity beyond the conservative setting. The proofs involve a careful study of minimal $\cW^u$-saturated sets and interact with recent results in the subject. 
\bigskip

\medskip

\end{abstract}

\section{Introduction}\label{SectionIntroduction}
An important property of a topological dynamical system is \emph{transitivity}. This asserts that the dynamics is strongly indecomposable: there exists a point whose forward orbit is dense in the entire space.
In this paper we study differentiable dynamics, so the dynamical system will be given by a diffeomorphism $f$ of a closed manifold $M$ (i.e. compact, connected, boundaryless).

An obvious obstruction to transitivity is the existence of a \emph{trapping region}: a non-empty open subset $U \neq M$ which satisfies $f(\overline{U}) \subset U$. For $C^1$-generic systems, thanks to Conley's theory and the connecting lemma for pseudo-orbits of \cite{BC-rec} one knows the following property:
\bigskip

\noindent
{\bf Generic dichotomy.} \emph{There exists a dense $G_\delta$ subset $\cG \subset \mathrm{Diff}^1(M)$ such that every $f \in \cG$ either is transitive or admits a trapping region.}
\bigskip

We remark that having a trapping region is an open property (even in the $C^0$-topology) but transitivity may be fragile (for a local version of this, see \cite{BCGP}). It is worth noticing that a stronger conclusion can be achieved, promoting transitivity to topological mixing, a stronger indecomposability property (see \cite{AC}). Existence of a trapping region is equivalent to the existence of a proper \emph{quasi-attractor} (see \S\ref{s.prelim}). In fact, every proper quasi-attractor is contained in a trapping region.

Having a result which is not robust has the drawback that in occasions, one is lead to work with other classes of diffeomorphisms which may be meager in $\mathrm{Diff}^1(M)$ (e.g. $C^2$-diffeomorphisms are meager in $\mathrm{Diff}^1(M)$, volume-preserving ones too, etc..). For this reason, one typically aims to promote generic properties to open and dense ones. The purpose of this paper is to make progress in this direction. 

It is well known that $C^1$-generic transitive diffeomorphisms admit weak forms of hyperbolicity (see \cite{BDP}) so it makes sense to restrict attention to diffeomorphisms having some weak form of hyperbolicity. Here, we will work with the strongest possible form beyond uniform hyperbolicity and obtain a quite satisfactory picture. 

We consider the (open) subset $\PH^1_{c=1}(M)$ of $C^1$-diffeomorphisms $f$ of $M$
which preserve a \emph{partially hyperbolic} decomposition, with a one-dimensional center.
More precisely, they preserve a splitting $T M = E^s \oplus E^c \oplus E^u$, $\dim(E^c)=1$, with the property that for some integer $\ell >0$ and for every unit vector  $v^\bullet \in E^\bullet_x$ (where $\bullet=s,c,u$) we have:
\begin{equation*}
\|Df_x^\ell v^s \| < \min \{1, \|Df_x^\ell v^c \|\} \leq \max \{ 1, \|Df_x^\ell v^c \| \} < \|Df_x^\ell v^u \|.
\end{equation*}
Our main result is:

\begin{teoA}
There exists an open and dense subset $\cO\subset \PH^1_{c=1}(M)$ such that if $f \in \cO$ then either $f$ is robustly transitive or there exists a non-empty open subset $U\subsetneq M$ such that $f(\overline U) \subset U$. 
\end{teoA}

The perturbation techniques we use to prove Theorem A all work in the conservative setting. We denote by $\PH^1_{c=1,\mathrm{vol}}(M)$ the subset of volume-preserving diffeomorphisms $f$ in $\PH^1_{c=1}(M)$. Among $\PH^1_{c=1,\mathrm{vol}}(M)$ those which are \emph{stably ergodic}\footnote{A volume-preserving diffeomorphism is \emph{stably ergodic} if any volume-preserving $C^2$-diffeomorphism in a $C^{1}$-neighborhood is ergodic with respect to the volume.} are dense \cite{BMVW}. However, this does not provide information about perturbations \emph{outside} the volume-preserving setting. To this end, we obtain the following result.

\begin{teoB}
There exists an open and dense subset $\cO_{\mathrm{vol}}\subset \PH^1_{c=1,\mathrm{vol}}(M)$ such that every $f \in \cO_{\mathrm{vol}}$ is robustly transitive in $\PH^1_{c=1}(M)$ (i.e. remains transitive under $C^1$-perturbations which may not preserve the volume). 
\end{teoB}

Theorem A of this paper was announced in \cite[\S 5]{CroPot} where a sketch of the proof was provided. We mention that the simplest context on which the results are new deals with partially hyperbolic diffeomorphisms with one dimensional center in the isotopy class of Anosov diffeomorphisms of $\TT^3$. Some specific partially hyperbolic examples in dimension 3 admit simpler proofs of analogous results (see e.g. \cite[\S 5.6]{BGHP}) but in general for 3-dimensional partially hyperbolic diffeomorphisms the results are new. See also \cite{JYang} for other results ensuring robust transitivity of volume-preserving partially hyperbolic diffeomorphisms. 

The proofs of the results are based on a careful study of $f$-invariant laminations saturated by strong unstable (or by strong stable) manifolds. It is a classical result (see \cite{HPS, CroPot}) that if $f$ is a partially hyperbolic diffeomorphism on $M$ then, there exist $f$-invariant foliations $\cW^s$ and $\cW^{u}$ tangent to $E^s$ and $E^u$ respectively. 

Since the results were announced, a new technique has been developed in \cite{ACW} that allows us to improve and simplify our results via the use of the SH property of \cite{PS-SH}. The use of that property already appeared in our initial approach by showing that under some assumptions, the foliations $\cW^s$ or $\cW^u$ are minimal (i.e. every leaf is dense in $M$). We refer the reader to \cite{CroPot} for an outline on the previous approach for which we are also filling the details here.  

We prove a more precise technical statement (Theorem \ref{T.SHorQA}) involving property SH that will be defined later. This allows to combine the results of \cite{CPoS} (that need to be strengthened, see Corollary \ref{C.QAorTrans}) and \cite{ACW} to obtain the following statement which is stronger than Theorems A and B.

\begin{teoC}
There is an open and dense subset $\cO\subset\PH^1_{c=1}(M)$ which decomposes as two disjoint open sets $\cO = \cO_m \cup \cO_a $ such that: 
\begin{itemize}
\item If $f\in\cO_m$ then either $f$ is a transitive Anosov or both $\cW^s$ and $\cW^u$ are robustly minimal.
In particular, $f$ is robustly topologically mixing\footnote{A diffeomorphism $f: M \to M$ is topologically mixing if for every pair of open sets $U,V$ there is $n_0>0$ so that $f^n(U)\cap V \neq \emptyset$ for every $n>n_0$.}. 
\item If $f \in \cO_a$ then the minimal $\cW^u$-saturated compact sets are proper and their number is finite.
Each of them either is contained in a connected component of some proper quasi-attractor. The same holds for minimal $\cW^s$-saturated compact sets and $f^{-1}$.
\end{itemize}
The set $\cO_{\mathrm{vol}}:=\cO_m\cap \PH^1_{c=1,\mathrm{vol}}(M)$ is open and dense in
$\PH^1_{c=1,\mathrm{vol}}(M)$.
\end{teoC}
\begin{obs} The proof (\S\ref{s.proofsAB}) gives additional properties:
\smallskip

\noindent
(a) Each minimal $\cW^u$-saturated compact set $\Lambda$ has some period $\tau\geq 1$. It satisfies both an $\varepsilon$-transversality property (Definition~\ref{def.transverse}) and the property SH (for $f^\tau$) unless it is the connected component of a uniformly hyperbolic attractor with contracting center bundle.
\smallskip

\noindent
(b) When $f\in \cO_m$ is Anosov, the minimal $\cW^u$-saturated and $\cW^s$ compact sets are unique. One of the foliations is minimal, but (using the $\varepsilon$-transversality) \cite{ACW} implies that if $f$ is also $C^2$ (or belongs to a dense G$_\delta$ subset of $\cO_m$ or $\cO_{\mathrm{vol}}$) then both $\cW^s$ and $\cW^u$ are minimal.
\smallskip

\noindent
(c) When $f\in \cO_a$ each connected component $K$ of a quasi-attractor contains one minimal $\cW^u$-saturated compact set $\Lambda$, which is unique. If $f$ belongs to a dense G$_\delta$ subset of $\cO_a$, then $\Lambda=K$.\smallskip

\noindent
(d) Note that when $f\in \cO_a$, minimal $\cW^u$-saturated sets are disjoint from minimal $\cW^s$-saturated sets.
\end{obs}

Theorem C improves \cite{BDU} (see also \cite{HHU} for higher dimensions) which shows that among partially hyperbolic robustly transitive diffeomorphisms with one dimensional center, the set of those for which one of the two strong foliations $\cW^s$ or $\cW^u$ is minimal is open and dense. 
Here, the improvement is that we do not assume a priori that the diffeomorphisms are robustly transitive (just that it is not approximated by a diffeomorphism with a proper trapping region) and we obtain robust minimality of both foliations simultaneously. It should be pointed out that neither \cite{BDU} nor \cite{HHU} are used in our proof.

As mentioned before, our results were known in some specific settings, but for instance, they were unknown in all generality for derived from Anosov diffeomorphisms (see \cite{HUY} for examples of derived from Anosov for which both strong foliations are minimal). 
In the setting of derived from Anosov diffeomorphisms of $\TT^3$, \cite{CPo} proves that
the dynamics is chain-transitive, so:

\begin{corD}
There is an open and dense subset $\cO_{DA}$ of the space of $C^1$ partially hyperbolic diffeomorphisms of $\TT^3$ isotopic to a linear Anosov automorphism such that for every $f \in \cO_{DA}$ both foliations $\cW^s$, $\cW^u$ are minimal. 
\end{corD}

Related results to the existence and uniqueness of Gibbs u-states and physical measures in dimension 3, which use the recent results from \cite{Katz,ALOS,EPZ}, will be studied in \cite{ACEPWZ}. 

\subsection{Techniques and some questions} 

Many parts of the result are independent of $C^1$-perturbations, so it is natural to ask (or even conjecture?):

\begin{quest} 
Do the same results hold in the $C^r$-topology? 
\end{quest}

The strategy of the proof of our main results is to try to develop a careful study of minimal $\cW^u$-saturated sets. It follows from \cite{CPoS} that there are finitely many such sets and that the $\cW^u$-leaves cannot be jointly integrable with the strong stable foliation.
By discussing carefully the possible geometric configurations, we are able to show that if
$f$ is not Anosov and if the minimal $\cW^u$-lamination is not contained in a proper trapping region, then, it must have some expanding property along the center and satisfies the SH property. This also produces some transversality which, combined with the recent results of \cite{ACW}, allows to prove minimality of the strong unstable foliation. Both the perturbation results of \cite{CPoS} and the analysis of proper attracting regions rely on perturbation techniques that are known only in the $C^1$-topology. 

Let us discuss the technique to obtain the center expansion along minimal $\cW^u$-saturated sets. The key idea is to understand the behavior transverse to such laminations: after taking the quotient by the direction $E^u$, we are reduced to a setting analogous to the case of a diffeomorphism preserving a splitting $E^s\oplus E^c$ with $\dim(E^c)=1$.
The strategy is then to apply techniques from \cite{PS-IHP, CPS} and prove that either the center bundle is contracting and $\Lambda$ is a (quasi)-attractor, or the center bundle is expanding, which gives the SH property for the lamination. In this paper, we are able to construct invariant transverse sections to the lamination which allow to directly apply \cite{PS-IHP}.

\subsection*{Organization of the paper}
\S\ref{s.prelim} is devoted to preliminaries about generic dynamics. In \S\ref{s.surface} we state results about dimension reduction and we explain how to apply the results from \cite{PS-IHP} in our context to deduce hyperbolicity.
In \S\ref{s.NJI} we explore the notions of transversality and non-joint integrability; we also state one of the main generic results, Corollary \ref{C.QAorTrans}. It is based on a general result (Theorem \ref{t.transverse-periodic}) that establishes transversality or existence of a periodic orbit for a general minimal unstable lamination. The other main generic result, Theorem \ref{T.SHorQA} allows to obtain the SH property. It is stated in \S\ref{s-SH} where we also present the results from \cite{ACW}.
The main theorems are then proved in \S\ref{s.proofsAB}. Finally in \S\ref{s.transverse} and \S\ref{s.sectiondynamics} we prove the generic results, Corollary \ref{C.QAorTrans} and Theorem \ref{T.SHorQA} respectively.
\medskip 
\medskip

{\small {\em Acknowledgements:} This paper grew out of an attempt to complete some preliminary results obtained by the first author with F. Abdenur who later abandoned the project. It is a pleasure to acknowledge his input and discussions about this subject. }

\section{Preliminaries}\label{s.prelim}
In all this paper $f$ denotes a partially hyperbolic diffeomorphism with $\dim E^c =1$ of a closed manifold $M$. At some points, we will also  assume that $f$ is volume-preserving. Some general references can be \cite{BDV,Crov-habilitation,CroPot}, where proofs of classical properties may be found.

\subsection{Partially hyperbolic notations}\label{ss.localproductstructure}

For $f \in \PH^1_{c=1}(M)$ there are $f$-invariant foliations $\cW^s, \cW^u$ tangent respectively to $E^s$ and $E^u$. We will write $\cW^s_f, \cW^u_f$, $E^s_f,E^c_f,E^u_f$ when we want to mention the dependence with respect to the diffeomorphism. We denote by $\cW^\bullet(x)$ the leaf of $\cW^\bullet$ ($\bullet=s,u$) through the point $x$. We can endow each leaf of $\cW^\bullet$ with the Riemannian metric  induced by its immersion and we denote by $\cW^\bullet_L(x)$ the set of points in $\cW^\bullet(x)$ which can be joined to $x$ by an arc of length smaller than $L$ inside $\cW^\bullet(x)$. Let $d_\bullet = \mathrm{dim}(E^\bullet)$. 

A center stable (resp. center unstable) disk is a disk of dimension $d_{cs}=d_s +1$ (resp. $d_{cu}=d_u+1$) tangent everywhere to a cone-field arbitrarily close to $E^s \oplus E^c$ (resp. $E^c \oplus E^u$).
Such disks with small diameter can be constructed using the exponential map. There are uniform distances $\delta_0,r_0>0$ such that if points $x,y$ are at distance less than $\delta_0$ then $\cW^s_{r_0}(x)$ (resp. $\cW^u_{r_0}(x)$) intersects at a unique point the center unstable (resp. center stable) disks centered at $y$ of size $r_0$. This property is sometimes called a \emph{local product structure}. We will fix such scales and denote $\cW^{\bullet}_{loc}(x):=\cW^{\bullet}_{r_0}(x)$. 

In general, for a partially hyperbolic diffeomorphism $f: M \to M$ we can (and will) assume that $M$ is endowed with an \emph{adapted metric} implying that there is $\lambda>1$ such that for every $v \in E^u$ we have $\|Df v \| \geq \lambda \|v\|$.
One can also assume that the metric is adapted with respect to the other invariant bundles (see \cite{CroPot}) but we will not use this. 

\subsection{Unstable laminations}\label{s.unstablelam}

An \emph{unstable lamination} $\Lambda$ is a compact non-empty set which is saturated by leaves of the foliation $\cW^u$. A \emph{minimal unstable lamination} $\Lambda$  is an unstable lamination which does not contain a proper unstable lamination.

\begin{defi}
An unstable lamination $\Lambda$ is called \emph{$f$-minimal} if it is $f$-invariant and if
$f$-invariant unstable laminations $\Lambda' \subset \Lambda$ coincide with $\Lambda$. 
\end{defi} 

Every unstable lamination contains at least one minimal unstable lamination and every $f$-invariant unstable lamination contains an $f$-minimal one. Note that being an $f$-invariant minimal unstable lamination is stronger than being an $f$-minimal unstable lamination. 

\begin{ej} 
Consider the diffeomorphism $f: \TT^2 \times S^1 \to \TT^2 \times S^1$ given by $f(x,t) = (g (x), R_\alpha(x))$ where $g: \TT^2 \to \TT^2$ is an Anosov diffeomorphism and $R_\alpha$ denotes the rotation of $S^1$ of angle $\alpha \in \RR$. If $\alpha \in \RR \setminus \QQ$ then $\TT^2 \times S^1$ is an $f$-minimal unstable lamination which is not a minimal unstable lamination. If $\alpha=0$ then the minimal unstable laminations coincide with the $f$-minimal ones and are of the form $\TT^2 \times \{t\}$. 
\end{ej}

\subsection{Minimality and recurrence} 
We will use basic properties about these laminations. The following properties are simple exercises and left to the reader (see \cite{CroPot}). 
We say that a set $\Gamma \subset M$ is $\eps$-dense on $\Lambda \subset M$ if it intersects
each ball $B_\eps(x)$ with $x \in \Lambda$. (We do not require $\Gamma\subset \Lambda$.)

\begin{lema}\label{p.fminimalit}
For every $f$-minimal unstable lamination $\Lambda$ and every $\eps,r>0$ there exists $n>0$ such that for every $d_u$-dimensional disk $D\subset \Lambda$ with internal radius larger than $r$ inside
a leaf of $\cW^u$, the union $D \cup f(D) \cup \ldots f^{n}(D)$ is $\eps$-dense on $\Lambda$. 
\end{lema}

\begin{cor}\label{c.minimaltransitive}
If $\Lambda$ is an $f$-minimal unstable lamination then $f|_\Lambda$ is transitive. If $\Lambda$ is a minimal unstable lamination, then $f|_\Lambda$ is topologically mixing.  
\end{cor}

\subsection{Quasi-attractors and $C^1$-generic dynamics}\label{ss.prelimgeneric}
Let $Q$ be a non-empty invariant compact set.
It is a \emph{transitive attractor} if it can be written as $Q= \bigcap f^n(U)$ where $U$ is an open set satisfying
$f(\overline U)\subset U$ and if $Q$ contains a dense forward orbit.
It is a \emph{quasi-attractor} if it can be written as
$Q= \bigcap_n U_n$ where $U_n$ are open sets satisfying $f(\overline{U_n}) \subset U_n$,
and if no proper subset $Q'\subsetneq Q$ has this property. (Equivalently it satisfies the following: given any $x\in Q$, a point $y$ belongs to $Q$ if and only if for any $\varepsilon>0$ there exists a $\varepsilon$-pseudo-orbit from $x$ to $y$.) The quasi attractor $Q$ is \emph{proper} if it is not equal to $M$.

\begin{obs} If $f$ is partially hyperbolic, then every quasi-attractor $Q$
is an unstable lamination.
Indeed, $Q$ is $f$-invariant, and for every $\delta>0$ the leaf $\cW^u(x)$ is the limit of the plaques $f^n(\cW^u_{\delta}(f^{-n}(x)))$. Given a trapping region $U \supset \Lambda$,
there is $\delta$ such that for every $z \in \Lambda$ we have $\cW^u_\delta(z) \subset U$, hence $\cW^u(x)\subset U$ for each $x\in \Lambda$. (See also \cite[Lemme 5.2]{BC-rec}.) In particular, every quasi-attractor contains at least one $f$-minimal unstable lamination. 
\end{obs}

The following result is essentially contained in \cite[\S{5}]{BC-rec} and gives a weak converse to the previous remark in the generic setting.

\begin{prop}\label{prop-QAgeneric}
There exists a dense $G_\delta$ subset $\cG \subset \mathrm{PH}^1_{c=1}(M)$ such that if $f \in \cG$ and if $\Lambda$ is an $f$-minimal unstable lamination which verifies one of the following properties:
\begin{itemize}
\item $\Lambda$ contains a periodic point $p$, or, 
\item the limit $\omega(y)$ is included in $\Lambda$ for every $y$ in a non-empty open set,
\end{itemize} \noindent then $\Lambda$ is a quasi-attractor.
\end{prop}
\begin{obs}
The same result holds among volume-preserving diffeomorphisms with a stronger conclusion (since there is no proper quasi-attractor): there is a dense $G_\delta$ subset $\cG \subset \mathrm{PH}^1_{c=1, \mathrm{vol}}(M)$ such that if $f \in \cG$ and if $\Lambda\neq M$ is an $f$-invariant lamination, then $\Lambda$ does not contain any periodic point. (For the perturbation results in the conservative setting, see \cite[\S 6]{BC-rec}.) 
\end{obs}

\begin{proof}
Let us first assume that $\Lambda$ contains a periodic point $p$.
The generic diffeomorphism $f$ is a continuity point of the map
$g \mapsto \overline{\cW^u_g(p_g)}$. If $\overline{\cW^u(p)}$ is not a quasi-attractor, then the connecting lemma for pseudo-orbits in~\cite{BC-rec} allows to build a sequence of diffeomorphisms $g_n\to f$ such that the sets $\overline{\cW^u_{g_n}(p_{g_n})}$
contain a same point $y\not\in \overline{\cW^u(p)}$; this contradicts the continuity at $f$.
This concludes the proof of the proposition under the first assumption.

Let us now assume that $\omega(y) \subset \Lambda$ for every $y$ in a non-empty open set. Then \cite[Corollaire 1.8]{BC-rec} concludes that $\Lambda$ contains a quasi-attractor $C$.
Since $C$ contains an $f$-minimal unstable lamination, we get $C=\Lambda$, concluding the proof in this case also.
\end{proof}

The next statement is~\cite[Theorem 15]{CP}. (The conservative version is simpler, see \cite{BC-rec}.)
\begin{teo}\label{t.palis}
Let $f$ be a diffeomorphism in a dense $G_\delta$ subset of $\mathrm{PH}^1_{c=1}(M)$. If $\Lambda$ is a quasi-attractor which is not uniformly hyperbolic, then it contains hyperbolic periodic points with stable dimension $\dim(E^s)$ and $\dim(E^s)+1$.
\end{teo}

\begin{cor}\label{c.palis}
In the setting of Theorem~\ref{t.palis}, there exists a uniformly hyperbolic set $K\subset \Lambda$ with expanding center bundle $E^c|_K$ and a non-empty open set $U$ of $\Lambda$ such that for any $y\in U$ the unstable leaf $\cW^u_{loc}(y)$ intersects
some $\cW^s_{loc}(x)$, $x\in K$.
\end{cor}
\begin{proof}
Let us consider two periodic points $p,q\in \Lambda$ with stable dimensions $\dim(E^s)+1$ and $\dim(E^s)$ respectively.
Since they belong to a common quasi-attractor, they are contained in a same chain-recurrence class.
The connecting lemma for pseudo-orbits in~\cite{BC-rec} implies that $\Lambda$ coincides with the chain-recurrence classes of $p$ and $q$. By $C^1$-perturbations one can create a blender, see~\cite[Section 6.2]{BDV}, which contains $q$ and is activated by $p$. This is a robust property and, up to restrict the dense G$_\delta$ subset, the diffeomorphism satisfies the following property: there exists $N\geq 1$ and a hyperbolic set $K\subset \Lambda$ containing $q$ such that for any point $y\in \Lambda$ close to $p$ the plaque $f^N\cW^u_{loc}(y)$ intersects some leaf $\cW^s_{loc}(x)$ with $x\in K$.
\end{proof}

\section{Subsets without stable connections}\label{s.surface} 

In this section we deal with invariant sets which do not have \emph{stable connections}, i.e. sets which intersect each strong stable leaves in at most one point. The following is proved in~\cite{BC-Center}.
\begin{teo}\label{teo-surface}
Let $f$ be a partially hyperbolic diffeomorphism and $K$ be a compact $f$-invariant set satisfying $\cW^s(x) \cap K = \{x\}$ for every $x \in K$. Then, there exists a submanifold $S$ such that:
\begin{itemize}
\item[--] The manifold $S$ contains $K$, is tangent to $E^{c}\oplus E^u$
and is locally invariant: $f(S)\cap S$ contains a neighborhood of $K$ in $S$.
\item[--] For any diffeomorphism $g$ $C^1$-close to $f$, any $g$-invariant compact set $K_g$ in a small neighborhood of $K$ satisfies $\cW^s_g(x) \cap K_g = \{x\}$ for every $x \in K_g$. Moreover $K_g$ is contained in
a submanifold $S_g$ which is arbitrarily $C^1$-close to $S$ as $g$ gets closer to $f$.
\end{itemize}
\end{teo}

When $f$ is $C^1$-generic and $\dim E^c=1$, this can be combined with the results of \cite{PS-IHP} (see also \cite[Corollary 2.31]{CP} or \cite{CPS}).

\begin{teo}\label{teo-hyp}
There is a dense $G_\delta$ subset $\cG \subset \mathrm{PH}^1_{c=1}(M)$ (resp. $\cG \subset \mathrm{PH}^1_{c=1,\mathrm{vol}}(M)$) such that if $f \in \cG$, then any compact $f$-invariant set $K \subset M$ satisfying $\cW^s(x) \cap K = \{x\}$ for every $x \in K$ is hyperbolic.

If $E^c|_{K}$ is uniformly expanding, then $K$ is a finite set of hyperbolic saddles.
If $E^c|_{K}$ is uniformly contracting and $K$ is an unstable lamination, then $K$ is
an attractor: there exists a trapping region $U$ such that $\bigcap_n f^n(U)=K$.
\end{teo}

\begin{proof}
Let us consider a countable basis of the topology of $M$, which consists in open sets
whose diameter is smaller than the scale  $\delta_0$ of the local product structure.
Let $\cB= \{U_n\}$ be the set of finite unions of elements of the base. Given $U_n \in \cB$
and a diffeomorphism $g$, we consider $K_n(g)$ the maximal invariant set of the closure of $U_n$. Let also $\cA_n \subset \mathrm{PH}^1_{c=1}(M)$ (resp. $\cA_n \subset \mathrm{PH}^1_{c=1,\mathrm{vol}}(M)$) be the open set of partially hyperbolic diffeomorphisms for which $K_n(g)$ is hyperbolic: $K_n(g)$ decomposes as a disjoint union of (possibly empty) invariant compact sets $K_n(g)=K_n^1(g) \cup K_n^2(g)$ such that $E^c$ is contracting on $K_n^1(g)$ and expanding in $K_n^2(g)$. Then $\cO_n := \cA_n \cup (\overline{\cA_n})^c$ is open and dense. We define $\cG = \bigcap_n \cO_n$ which is $G_\delta$ and dense.

\begin{af}
For any $f \in \cG$ and $U_n \in \cB$, if the set $K_n(f)$ satisfies $\cW^u_f(x) \cap K_n(f) = \{x\}$ for every $x \in K_n(f)$ then it is hyperbolic.
\end{af}

This concludes. Indeed if $f \in \cG$ and if $K$ is an $f$-invariant compact set which has no stable connection, we consider some $U_n\in \cB$ which is a small neighborhood of $K$.
By Theorem~\ref{teo-surface}, the set $K_n(f)$ has no stable connection, hence, by the claim, is hyperbolic. One deduces that $K\subset K_n(f)$ is hyperbolic.
When $E^c|_K$ is uniformly expanded,
the set $K$ is a uniformly expanded invariant set for the induced system $(S,f)$,
hence it is a finite union of sources for $f|_S$.
When $E^c|_K$ is uniformly contracted and $K$ is an unstable lamination,
it is a hyperbolic set saturated by its unstable manifolds for the induced system $(S,f)$,
hence is an attractor; since $f$ uniformly contracts transversally to $S$, the set $K$
is also an attractor for $(M,f)$.

\smallskip

It remains to prove the claim: we will show that there are arbitrarily $C^1$-small perturbations $g$ of $f$ such that $K_n(g)$ is hyperbolic; since $f$ belongs to the open set  $\cA_n \cup (\overline{\cA_n})^c$ this implies $f \in \cA_n$ and thus $K_n(f)$ is hyperbolic.

We first explain how to perturb in the non-conservative case and then how to adapt it to the conservative setting. Using Theorem \ref{teo-surface} we know that $K_n(f)$ is contained in a locally invariant submanifold $S$ for $f$ tangent to $E^c \oplus E^u$ at $K_n(f)$. By perturbation we can assume without loss of generality that $f\in \cA_n \cup (\overline{\cA_n})^c$ is smooth. By local invariance, there exists a neighborhood $V$ of $K_n(f)$ such that $f(S)\cap V\subset S$.
We can consider a sequence of smooth submanifold $S_k$ which converges to $S$ in the $C^1$-topology (and can be seen as a graphs over $S$ with respect to the projection along $\cW^s$). For each $k$ there exists a smooth diffeomorphism $h_k$ such that $h_k\circ f(S_k)\cap V\subset S_k$. One can furthermore require that $(h_k)$ converges to the identity in the $C^1$-topology.

We then set $g_k=h_k\circ f$. The maximal invariant set $K_{n}(g_k)$ is close to $K_n(f)$ and has to be contained in $S_k$ (because points outside $S_k$ separate from $K_n(g_k)$ by forward iterations up to a
uniform distance). The diffeomorphism $g_k$ can be chosen so that inside $S_k\cap V$ all periodic orbits are hyperbolic and each invariant one-dimensional compact submanifold contains a periodic orbit. It follows that one can apply the results of \cite{PS-IHP} to deduce that the set $K_n(g_k)$ is hyperbolic, concluding the proof. 

In the conservative case, the difficulty is to show that one can choose the perturbations $g_k$ as above to be still conservative. First, we can apply~\cite{Avila} and reduce to the case where the initial diffeomorphism $f$ is conservative and smooth. Note that there is no connected component $S'$ of $S$ which is a compact submanifold without boundary such that $K_n(f)\cap S'=S'$: this would imply that $S'$ is fixed by an iterate of $f$ and normally contracted, contradicting the volume-preservation. In particular the one can choose the neighborhood $V$ to be a submanifold with boundary such that $M\setminus V$ is connected. We build as before diffeomorphisms $h_k$ but require in addition that $h_k\circ f$ preserves the volume form on $V$. By applying Moser's trick (see~\cite[Theorem 5.1.27]{KH})
we can modify $h_k$ outside $V$ so that it is volume-preserving everywhere and still satisfies that $h_k\to \text{Id}$
for the $C^1$-topology. In particular $g_k$ is conservative too.
\end{proof}


\section{Joint integrability and transversality}\label{s.NJI}

\subsection{Definitions}\label{ss.def-transvserse}
We first discuss stable holonomy maps.

Given two points $x, y \in M$ which belong to the same stable manifold $\cW^s(x)=\cW^s(y)$ and small disks $D_x,D_y$ transverse to $E^s$ of dimension $d_{cu}$ and containing $x,y$ respectively, we consider a holonomy map $\Pi^s_{x,y} : U \subset D_x \to D_y$ along the leaves of $\cW^s$, where $U$ is a small neighborhood of $x$ that depends on $D_y$ (and the choice of a path from $x$ to $y$). For general foliations, the holonomy map may depend on the choice of an homotopy class of path $\gamma\subset \cW^s(x)$ connecting $x$ to $y$.
In our setting the leaves of $\cW^s$ are uniformly contracted by $f$, hence they are simply connected.
Consequently the homotopy class is unique and the holonomy map $\Pi^s_{x,y}$ is well-defined (i.e. given two paths, the holonomy coincides in the intersection of their domains).

There is a strong notion of not being jointly integrable which is implicit in \cite{CPoS}
(note that it includes laminations with no stable connections):

\begin{defi}\label{def.NJI}
We say that an unstable lamination $\Lambda$ is \emph{strongly non-jointly integrable} if for any $x\neq y$ in  $\Lambda$ with $\cW^s(x)=\cW^s(y)$ and for any transverse discs $D_x,D_y$
containing neighborhoods of $x$ and $y$ in $W^u_{loc}(x)$ and $W^u_{loc}(y)$,
the associated holonomy map $\Pi^s_{x,y}$ has the following property:
there exists $z\in \cW^u_{loc}(x)$ arbitrarily close to $x$ such that $\Pi^s_{x,y}(z) \notin \cW^u_{loc}(y)$.
\end{defi}

When $\dim E^c =1$ we can define another related notion which has been studied in \cite{ACW}. This definition is implicit in \cite{CP}. See also  \cite{CroPot}. 

\begin{defi}\label{def.transverse}
Given $\varepsilon>0$ small,
an unstable lamination $\Lambda$ is  $\eps$-\emph{transverse} if there exists $R>0$ such that every unstable disk inside $\Lambda$ with internal radius $R$ contains points
$x\neq y$ in the same local stable manifold so that the holonomy $\Pi^s_{x,y}$ maps $\cW^u_{\eps}(x)$ into both connected components of $D_y \setminus \cW^{u}_{loc}(y)$.  We say that $\Lambda$ is (\emph{locally}) transverse if it is $\eps$-transverse for all small $\eps$. 
\end{defi}

\begin{figure}[htbp]
\begin{center}
\includegraphics[scale=0.5]{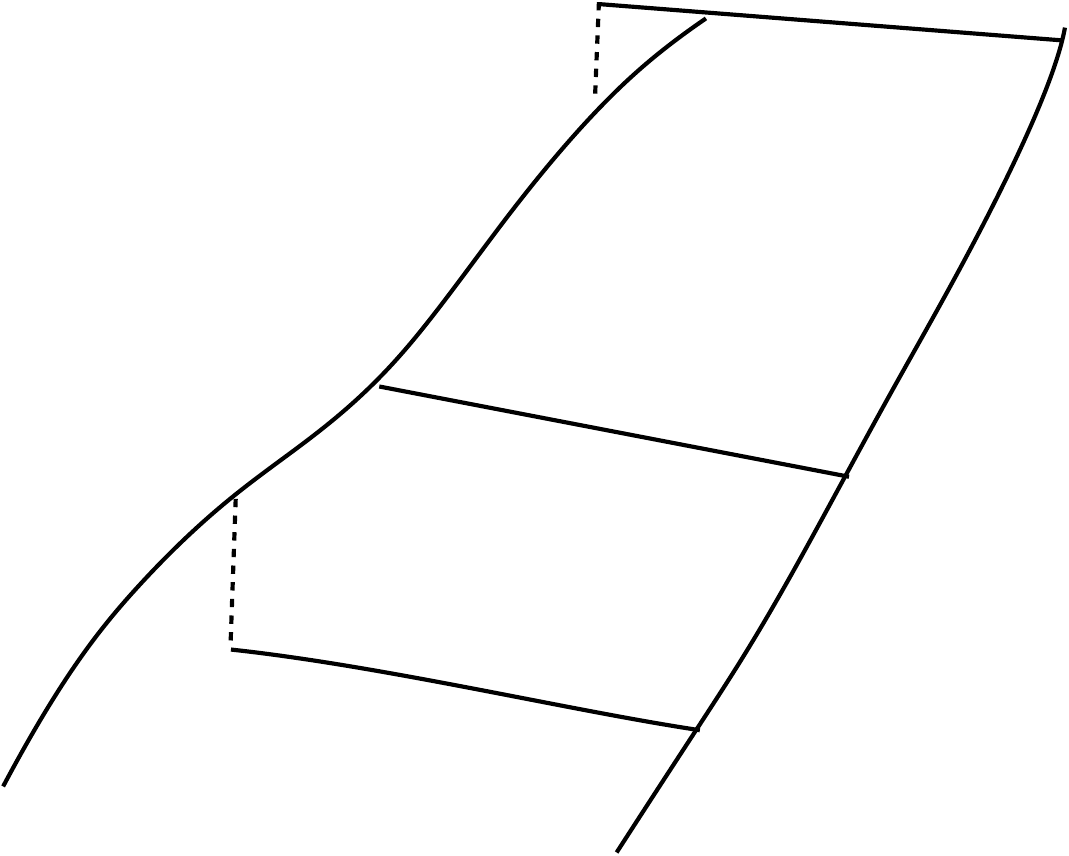}
\begin{picture}(0,0)
\put(-57,79){$y$}
\put(-250,16){\small$\cW^{u}(x)$}
\put(-100,7){\small$\cW^{u}(y)$}
\put(-180,115){$x$}
\put(-124,90){$\cW^{s}(x)$}
\end{picture}
\end{center}
\vspace{-0.5cm}
\caption{{\small Transversality of the foliation $\cW^u$.}\label{f.transverse}}
\end{figure}
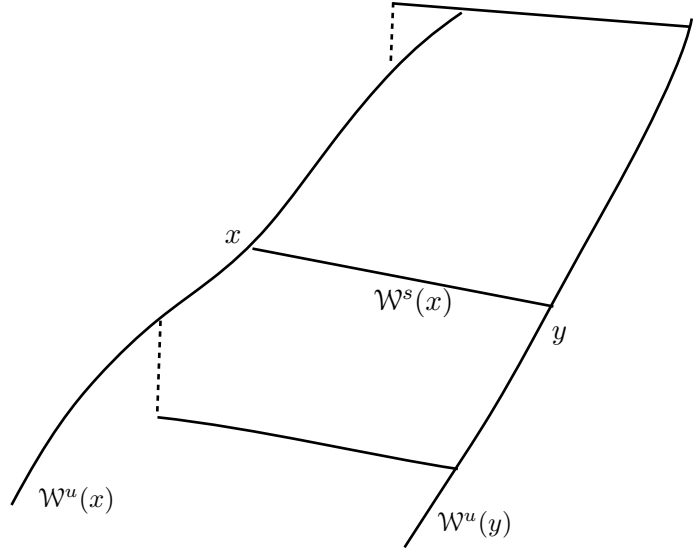

We will only use this notion of transversality, so we will omit to mention the ``locally'' when referring to this concept. In \cite{ACW} a notion of global transversality is introduced which is relevant in many cases but that we shall not need. Note that $\eps$-transversality is a $C^1$-open property, while local transversality is a priori not.

\subsection{Finiteness}\label{ss.finiteness}

We state a consequence of \cite{CPoS}.

\begin{teo}\label{T.RobustSH}
There is a dense $G_\delta$ subset $\cG \subset \PH^1_{c=1}(M)$  (resp. $\cG_{\mathrm{vol}} \subset \PH^1_{c=1, \mathrm{vol}}(M)$) such that any $f \in \cG$ (resp. $f \in \cG_{\mathrm{vol}}$) has finitely many $f$-minimal unstable laminations; they are strongly non-jointly integrable.

Moreover, $f$ admits a neighborhood $\cU\subset\PH^1_{c=1}(M)$  (resp. $\PH^1_{c=1, \mathrm{vol}}(M)$)
such that any $g\in \cU$ has the same number of $f$-minimal unstable laminations as $f$,
denoted by $\Lambda_{1,g},\dots,\Lambda_{k,g}$. The maps $g\mapsto \Lambda_{i,g}$ are semi-continuous for the Hausdorff topology and are continuous at $f$.
\end{teo} 

\begin{proof}
By \cite[Theorem C] {CPoS}, there is a $C^1$-open and dense set $\cO$ in $\PH^{1}_{c=1}(M)$ such that for $f \in \cO$ there are finitely many minimal unstable laminations $\Lambda_1, \ldots, \Lambda_k$.
Note that the following semi-continuity holds: for any $\varepsilon>0$ and any $g$ $C^1$-close to $f$,
each unstable lamination $\Lambda$ of $g$ is $\varepsilon$-dense on one of the laminations
$\Lambda_i$ of $f$. Moreover the arguments in \cite[\S 2.5]{CPoS} are uniform with respect to the diffeomorphism
and show that for any lamination $\Lambda_i$ there exists $\varepsilon>0$ such that a diffeomorphism $g$ close to $f$ admits at most one minimal unstable lamination which is $\varepsilon$-dense on $\Lambda_i$.
Consequently the number of minimal unstable lamination is a semicontinuous function, hence is locally constant on an open and dense subset $\cO' \subset \cO$.

Each $g$ close to $f$ admits exactly one minimal unstable lamination
which is $\varepsilon$-dense on $\Lambda_i$, it is denoted by $\Lambda_{i,g}$.
As noticed above, the maps $g\mapsto \Lambda_{i,g}$ are lower semi-continuous for the Hausdorff topology.
Therefore, there is a dense G$_\delta$ subset $\cG\subset\cO'$ where these maps are continuous.
The strong non-joint integrability is given by \cite[Theorem 2.1]{CPoS}.

The proof in the volume-preserving case is the same as the results \cite{CPoS} are obtained via conservative perturbations.
\end{proof}

The following is \cite[Proposition 1.5]{ACW} and is a variant of~\cite{CPoS}.

\begin{prop}\label{prop-njifinite}
If an $f$-minimal unstable lamination is transverse, then it is a finite union of minimal unstable laminations. 
\end{prop}

\begin{obs}\label{r-njifinite}
A similar statement holds if one replaces the transversality by the strong non-joint integrability, see \S\ref{s-njifinite}
\end{obs}

\subsection{Criteria for transversality}
The state~\cite[Proposition 1.4]{ACW}:
\begin{prop}\label{p.transverse0}
If the foliation $\cW^u$ is $f$-minimal, then it is transverse.
\end{prop}

The next results will be proved in \S\ref{s.transverse}.  

\begin{teo}\label{t.transverse-periodic}
If an $f$-minimal unstable lamination is strongly non-jointly integrable and has no periodic points, then it is transverse. 
\end{teo}

\begin{cor}\label{C.QAorTrans}

There is a dense $G_\delta$ subset $\cG\subset \PH^1_{c=1}(M)$ such that, for $f \in \cG$, any $f$-minimal unstable lamination is transverse or is a proper uniformly hyperbolic transitive attractor whose center direction is attracting.

For any $\varepsilon>0$, there is a dense open subset $\cO_{\mathrm{vol}} \subset \PH^1_{c=1, \mathrm{vol}}(M)$ such that, for $f \in \cO_{\mathrm{vol}}$, the unstable foliation $\cW^u$ is $\varepsilon$-transverse. 
\end{cor}

\subsection{Transverse structure} 

An unstable lamination $\Lambda$ is \emph{trans\-versally Cantor} if
its intersection with any disc transverse to $E^u$ is totally disconnected. 

\begin{prop}\label{prop-lamtransvcantor}
If $\Lambda$ is an $f$-minimal unstable lamination which is strongly non-jointly integrable and not transversally Cantor, then there is a non-empty open set $U \subset M$ such that $\omega(x) \subset \Lambda$ for every $x \in U$.
\end{prop}

\begin{proof}
Let $D$ be a disk transverse to $E^u$ such that $D \cap \Lambda$ contains some non-trivial connected set $\cC$ and let us fix some $x_0 \in \cC$. We take a center stable disk $\cD^{cu}$ containing the local strong unstable manifold $\cW^{u}_{loc}(x_0)$ and transverse to $E^s$.
For each $t\in \cC$ close to $x_0$ we project the local unstable manifold $\cW^{u}_{loc}(t)$ to $\cD^{cu}$
by $\cW^s$-holonomy and get a continuum $\Gamma_t\subset \cD^{cu}$.

Let us consider some $C^1$ arc $I\subset \cD^{cu}$, connecting two points $p,q$, transverse to $\cW^{u}_{loc}(x_0)$
and which intersects $\Gamma_{x_0}=\cW^{u}_{loc}(x_0)$ only at $x_0$. It is denoted by $I=[p,q]$. For each $t\in \cC$, we denote by $\cI(\Gamma_{t},[p,q])$ the intersection number mod 2 between
$\Gamma_t$ and $[p,q]$. By construction $\cI(\Gamma_{x_0},[p,q])=1$, hence by reducing the connected set $\cC$,
one gets $\cI(\Gamma_t, [p,q])=1$ for all $t \in \cC$.

Let  $B \subset \cD^{cu}$ be some small neighborhood of $x_0$ and for each $y\in B$, let $[p,y]$ and $[y,q]$
be some smooth arcs joining the points $p,y,q$, whose concatenation is close and homotopic to $[p,q]$ (with fixed endpoints). The stability properties of the intersection number give:
\begin{af}
For every $y \in B$ such that $y \notin \Gamma_t$, the numbers $\cI(\Gamma_t, [p,y])$ and $\cI(\Gamma_t, [y,q])$ are well defined, their sum equals $1$. They vary continuously with the pairs $(y,t) \in B\times \cC$ such that $y \notin \Gamma_t$.
\end{af}
\begin{proof}
For every small smooth approximation of $\Gamma_t$ (not necessarily embedded) the intersection number mod 2 is well-defined. It is invariant under homotopy so it is independent from the approximation and from the arc joining $p$ and $q$. The continuity also follows from the invariance under homotopy. 
\end{proof}

From the strong non-joint integrability there is $c\in \cC$ close to $x_0$ and
a non-empty open set $V\subset \Gamma_{x_0}\cap B$ which is disjoint from $\Gamma_{c}$. Consider then an arc $J$ which is obtained by deforming $I$ inside $B$ and such that it intersects $\Gamma_{x_0}$ transversally and uniquely in some point $z \in V$. Without loss of generality (up to reversing the roles of $p,q$), we can assume that $[z,q]$ is disjoint from $\Gamma_c$.  It follows that close to $z$ there is some point $z' \in J$ such that $\cI(\Gamma_{x_0},[z',q]) + \cI(\Gamma_c, [z',q]) = 1$. Hence the set defined by
$$R:= \{w \in B \ : \ \cI(\Gamma_{x_0},[w,q]) + \cI(\Gamma_c, [w,q]) = 1\}$$ is non-empty and open inside $\cD^{cu}$.

We claim that $R \subset \bigcup_{t \in \cC} \Gamma_t$. Indeed if a point $w_0 \in R$ does not belong to the union $\bigcup_{t \in \cC} \Gamma_t$, then $t\mapsto  \cI(\Gamma_t, [w_0,q])$ is well defined, continuous on
$\cC$, takes the two values $0$ and $1$, contradicting the connectedness of $\cC$.
It follows that $\bigcup_{w \in R} \cW^{s}_{loc}(w)$ has non-empty interior in $M$. The $\omega$-limit set of each of its point is contained in $\Lambda$.
\end{proof}

\section{The SH property}\label{s-SH}

\subsection{Definition and first properties} 

Extending the definition of \cite{PS-SH}, we say that an $f$-invariant unstable lamination $\Lambda$ has the \emph{SH property} if there exists $N>0$ and $\eps>0$ such that for every unstable disk $D^u \subset \Lambda$ with inner radius larger than $\eps$ there exists $D' \subset D^u$ sub-disk such that:

\begin{itemize}
\item[--] $f^N(D')$ contains a disk of inner radius larger than $\eps$;
\item[--] every $x\in D'$ and $v \in E^c(x) \setminus \{0\}$ satisfy $\|Df^N v \| > 2 \|v\|$.
\end{itemize}
A stable lamination has the SH property if it has the SH property for $f^{-1}$.
\begin{obs}\label{r.SH-robust}
If $g$ is $C^1$-close to $f$, then any $g$-invariant unstable lamination $\Lambda_g$ contained in a neighborhood of $\Lambda$ also has the SH property.
\end{obs}

The robustness of the minimality of the strong foliations $\cW^s,\cW^u$ can be obtained from the following result, which is a direct consequence of \cite{PS-SH} (see also \cite{CroPot,ACW}).

\begin{prop}\label{prop-robustminimal}
Let $f \in \mathrm{PH}^1_{c=1}(M)$ such that both foliations $\cW^s$, $\cW^u$ are minimal and have the SH property. Then $\cW^s$ and $\cW^u$ are robustly minimal. In particular, $f$ is robustly topologically mixing.
\end{prop}

\subsection{A criterion for minimality}
One way to obtain the robust transitivity is to prove that all minimal sets are transverse and have the SH property. Both properties are robust and, after perturbation, they allow to improve $\eps$-density of the orbits to density using the fact that every small disk tangent to $E^{c} \oplus E^u$ will reach size larger than $\eps$ (compare to \cite[Theorem 1.2]{HT} or \cite[Proposition 5.14]{CroPot}). However, we will instead use the following stronger result from \cite{ACW} to obtain minimality of the strong foliations:

\begin{teo}
\label{teo-ACW} 
Let $f \in \PH^2_{c=1}(M)$ be a $C^2$ diffeomorphism and let $\eps>0$ be smaller than $r_0$ (cf. \S\ref{ss.localproductstructure}). If $\Lambda$ is an $\eps$-transverse $f$-invariant unstable lamination with the SH-property, then it contains a disk $D$ tangent to $E^c \oplus E^u$. If $f$ is transitive, then $\Lambda=M$. 
\end{teo}

This theorem, under a weaker notion of transversality, is \cite[Theorem C]{ACW}, but the statement above will be enough for our purposes. Note that the second part of the conclusion follows from the first, because the union of the stable manifolds of points in $D$ has non-empty interior; hence any point with a dense forward orbit belongs to a stable manifold of $\Lambda$, proving that there is some point in $\Lambda$ which has dense forward orbit. Since $\Lambda$ is closed and $f$-invariant, it must be equal to $M$. 

\begin{obs}\label{r.disk-uniform}
It follows from Theorem \ref{teo-ACW} and the definition of the SH property that
the disk $D$ may be chosen with a size which only depends on:
\begin{itemize}
\item the angle between the bundles of the partially hyperbolic splitting, 
\item the domination constants and contraction/expansion on the bundles,
\item the $C^1$-norm of $f$, 
\item the strength of the SH-property.
\end{itemize}
\end{obs}

With this in mind, one obtains the following:

\begin{cor}\label{cor-ACW}
Let $f \in \PH^1_{c=1}(M)$ and $\Lambda$ be a transverse $f$-invariant unstable lamination with the SH-property.
Let us assume that there exists a sequence of $C^2$ diffeomorphisms $(g_n)$ which converges to $f$ in the $C^1$-topology and which admit $g_n$-invariant unstable laminations $\Lambda_g$ which converge to $\Lambda$ for the Hausdorff topology. Then, $\Lambda$ contains a disk $D$ tangent to $E^c \oplus E^u$. If $f$ is transitive then $\Lambda=M$. 
\end{cor}

\begin{proof}
Let us consider a sequence of $C^2$ diffeomorphisms $(g_n)$ which converges to $f$ in the $C^1$-topology.
Note that the number $r_0$ satisfies the properties in \S\ref{ss.localproductstructure} for all the diffeomorphisms $g_n$. Let us fix $\varepsilon>0$. For $n$ large, the Lamination $\Lambda_{g_n}$ is $\varepsilon$-transverse (see \S\ref{ss.def-transvserse}) and satisfies the SH property (by Remark~\ref{r.SH-robust}). Hence $\Lambda_{g_n}$ contains a disc $D_n$ tangent to $E^s_{g_n} \oplus E^{c}_{g_n}$.
By Remark~\ref{r.disk-uniform} these disks may be chosen with a uniform size.
Since $\Lambda_{g_n} \to \Lambda$ in the Hausdorff topology, one concludes that $\Lambda$ must also contain a disk $D$ tangent to $E^s \oplus E^c$. 
The last part of the conclusion of the corollary follows from the existence of the disk as explained above.
\end{proof}

\subsection{A generic property} 

The next result will be proved in \S\ref{s.sectiondynamics} and will allow us to apply Corollary \ref{cor-ACW}:

\begin{teo}\label{T.SHorQA} 
There exists a dense $G_\delta$ subset $\cG \subset \PH^1_{c=1}(M)$  (resp. $\cG_{\mathrm{vol}} \subset \PH^1_{c=1, \mathrm{vol}}(M)$) such that, for $f \in \cG$  (resp. $f \in \cG_{\mathrm{vol}}$), any $f$-minimal unstable lamination $\Lambda$ satisfies one of the following properties:
\begin{itemize}
\item either $\Lambda$ is a uniformly hyperbolic transitive attractor (possibly $\Lambda=M$) and its center direction is attracting,
\item or $\Lambda$ has the SH property. 
\end{itemize} 
\end{teo}

Note that in the volume preserving case the first possibility can only happen if $f$ is Anosov (i.e. $\Lambda=M$).

\section{Proof of the main results}\label{s.proofsAB}

In this section we establish the proof of Theorem C assuming the results stated in the previous sections.
Theorem A is a direct corollary of Theorem C.
Since Theorem B deals perturbations which may not be volume-preserving, we will give a short proof at the end of this section. 

\subsection{Proof of Theorem C (dissipative case)} \label{ss.dissipative}
We consider $C^1$-generic diffeomorphism $f$: it belongs to a dense $G_\delta$ subset $\cG \subset \PH^1_{c=1}(M)$ such that Proposition \ref{prop-QAgeneric},  Corollary \ref{C.QAorTrans}, and Theorems \ref{t.palis}, \ref{T.RobustSH} and \ref{T.SHorQA} hold. 

Let $\Lambda$ be an $f$-minimal saturated set. By Theorem \ref{T.RobustSH}, it admits a continuation $\Lambda_g$ for any $g$ $C^1$-close to $f$. As a preliminary discussion, we distinguish two settings:
\begin{itemize}
\item If $\Lambda$ is not transverse or does not have the SH property, then Corollary \ref{C.QAorTrans} and Theorem \ref{T.SHorQA} imply that $\Lambda$ is a uniformly hyperbolic transitive attractor such that $E^c$ is contracted.
The unstable manifolds of $\Lambda$ coincide with the leaves of $\cW^u$, hence the classical theory of hyperbolic attractors shows that $\Lambda$ is a finite union of minimal unstable laminations which are cyclically permuted by $f$.
\item If $\Lambda$ is transverse and has the SH-property, then by Corollary \ref{cor-ACW} it contains a disc tangent to $E^c \oplus E^u$. Hence, there is a non-empty open set of points in $M$ whose $\omega$-limit set is contained in $\Lambda$. Proposition \ref{prop-QAgeneric} then implies that $\Lambda$ is a quasi-attractor.
Moreover Proposition~\ref{prop-njifinite} implies that $\Lambda$ is a finite union of minimal unstable laminations which are cyclically permuted by $f$.
\end{itemize}
We are now able to consider the two cases of the dichotomy in the theorem.
\medskip

\paragraph{\bf First case: $f$ is chain-transitive.}
Then $M$ is a quasi-attractor and $M$ is an $f$-minimal unstable lamination, which is a finite union of minimal unstable laminations. Since $M$ is connected, the unstable lamination $\cW^u$ is minimal. The same applies to $f^{-1}$ and the stable lamination $\cW^s$ is also minimal.
If $f$ is not Anosov, then both $\cW^s$ and $\cW^u$ have the SH property and are transverse:
Proposition \ref{prop-robustminimal} shows that both foliations $\cW^s,\cW^u$ are robustly minimal
and $f \in \cO_m$
If $f$ is Anosov, this is also the case for all its $C^1$-perturbations and $f \in \cO_m$ also.

\begin{obs}\label{r.uniqueness-minimal}
In the case $f$ is a transitive Anosov (for instance if $E^c$ is uniformly expanded), then one of the strong foliations is robustly minimal (here $\cW^s$). We claim that for any diffeomorphism $g$ that is $C^1$-close, the other lamination (here $\cW^u$) contains a unique minimal sublamination. Indeed as $\cW^u$ is minimal and transverse for $f$, there exists $R>0$ such that for any $g$ close to $f$ and any $x,y\in M$, the leaves $\cW^u_R(x)$ and $\cW^u_R(y)$ intersect a same stable leaf. This implies that any two minimal unstable laminations intersect, hence coincide.
See also the proof of \cite[Proposition 1.5]{ACW}.
\end{obs}

\paragraph{\bf Second case: $f$ is not chain-transitive.}
By Theorem \ref{T.RobustSH} there are finitely many $f$-minimal unstable laminations.
From the discussion above, they coincide with the proper quasi-attractors, and the minimal unstable laminations are the connected components of the quasi-attractors. Consequently, there are finitely many minimal unstable laminations, 
$\Lambda_1,\dots,\Lambda_k$, each of them is periodic. Moreover the semi-continuity in Theorem \ref{T.RobustSH}
implies that for any diffeomorphism $g$ that is $C^1$ close to $f$, any minimal unstable lamination is Hausdorff close to one of the $\Lambda_i$.
We will consider each of them separately. In the following we fix $\Lambda_i$ and denote by $\tau$ its period.

If $L:=\Lambda_i\cup f(\Lambda_i)\cup \dots\cup f^{\tau-1}(\Lambda_i)$ is uniformly hyperbolic with contracting center bundle, then for any diffeomorphism $g$ close to $f$ any minimal unstable lamination $\Lambda$ for $g$ that is Hausdorff close to $\Lambda_i$ coincides with the hyperbolic continuation of $\Lambda_i$, hence it is still the connected component of a proper uniformly hyperbolic transitive attractor.

If $L=\Lambda_i\cup f(\Lambda_i)\cup \dots\cup f^{\tau-1}(\Lambda_i)$ is uniformly hyperbolic with expanding center bundle, 
then it is transverse and arguing as in Remark~\ref{r.uniqueness-minimal}, for any diffeomorphism $g$ close to $f$,
there exists a unique minimal unstable lamination for $g$ in a neighborhood of $\Lambda_i$. It is contained in the hyperbolic continuation of $\Lambda_i$, which is a connected component of a proper uniformly hyperbolic transitive attractor.

If $L=\Lambda_i\cup f(\Lambda_i)\cup \dots\cup f^{\tau-1}(\Lambda_i)$ is not uniformly hyperbolic,
then by Theorem~\ref{t.palis}, $\Lambda_i$ contains a hyperbolic periodic point $p$ whose stable space is
$E^s(p)\oplus E^c(p)$. As a consequence for any diffeomorphism $g$ that is $C^1$-close to $f$,
any minimal unstable lamination $\Lambda_g$ in a neighborhood of $\Lambda_i$ intersects the stable set of the continuation $p_g$, hence coincides with $\overline {\cW^u(p_g)}$. In particular such a minimal unstable lamination is unique. Since $L$ has arbitrarily small trapping neighborhoods,
for any $g$ close to $f$ there exists a quasi-attractor with a connected component close to $\Lambda_i$;
such a component contains a minimal unstable lamination, which then has to coincide with $\Lambda_g$ by uniqueness.

We have proved that for any diffeomorphism $g$ close to $f$: (1) any minimal unstable lamination is contained in the connected component of a quasi-attractor
and is close to some $\Lambda_i$; (2) any connected component of a quasi-attractor contains a minimal unstable lamination which is unique.
In particular $g \in \cO_a$.
\medskip

\paragraph{\bf Conclusion.}
We have proved that the open set $\cO_m\cup \cO_a$ is dense in $\PH^1_{c=1}(M)$.
By Corollary~\ref{c.minimaltransitive} the diffeomorphisms in $\cO_m$ are robustly topologically mixing.
This completes the proof of Theorem C in the dissipative case.

\subsection{Proof of Theorem C (volume-preserving case) and Theorem~B}
The proof is similar to the dissipative case. By Corollary~\ref{C.QAorTrans} and Theorem~\ref{T.SHorQA}, there exists an open and dense subset $\cO^1_\mathrm{vol}\subset \PH^1_{c=1, \mathrm{vol}}(M)$ of diffeomorphisms $f$ such that: (1) both $\cW^s$ and $\cW^u$ are $\varepsilon$-transverse
for some $\varepsilon>0$ smaller than $r_0$; (2) either $f$ is Anosov or both $\cW^s$ and $\cW^u$ satisfy the SH property. By Theorem~\ref{teo-ACW} and Proposition~\ref{prop-njifinite}, for any $C^2$-diffeomorphism $f\in \cO^1_\mathrm{vol}$
both foliations $\cW^s$, $\cW^u$ are minimal. By Proposition~\ref{prop-robustminimal}, there exists a $C^1$-neighborhood $\cO_\mathrm{vol}\subset \PH^1_{c=1, \mathrm{vol}}(M)$ of the set of $C^2$ diffeomorphisms in $\cO^1_\mathrm{vol}$ such that for any $f\in \cO_\mathrm{vol}$ which is not Anosov both strong foliations $\cW^s,\cW^u$ are robustly minimal. Hence any diffeomorphism in $\cO_\mathrm{vol}$ satisfies the conclusions of Theorem C.

Let us consider again any diffeomorphism $f\in \cO_\mathrm{vol}$.
If $f$ is Anosov, it is structurally stable and thus robustly topologically mixing in the sense that small $C^1$-perturbations (conservative or not) remain topologically mixing.
Otherwise both foliations are robustly minimal, and by Corollary~\ref{c.minimaltransitive} we get that $f$ is robustly topologically mixing.


\section{Transversality}\label{s.transverse}

Here we show Theorem \ref{t.transverse-periodic} and Corollary \ref{C.QAorTrans}.
Theorem \ref{t.transverse-periodic} is a direct consequence of Propositions~\ref{prop-lyapunovexponents} and~\ref{p.transverse} proved in the next sections. In particular we obtained the following stronger dichotomy for an $f$-minimal lamination $\Lambda$ which is strongly non-jointly integrable:

\begin{itemize}
\item Either $\Lambda$ contains a dense set of periodic points with negative center Lyapunov exponent.
\item Or there is a strong stable manifold which intersects $\Lambda$ in at least three points and $\Lambda$ is transverse. 
\end{itemize}

\subsection{Lyapunov exponents and entropy}\label{ss.prelimlyap}

The following result is based on \cite{ABC} and entropy considerations (compare with \cite[Proposition 7.4]{CPoS}). 

\begin{prop}\label{prop-lyapunovexponents}
Let $\Lambda$ be an $f$-minimal unstable lamination. Then there exists a dense set of points $z_0\in \Lambda$ such that, either $\cW^s_{loc}(z_0)\cap \Lambda$ is uncountable, or $z_0$ is periodic with a negative center Lyapunov exponent.
\end{prop}

\begin{proof}
Let us consider an ergodic u-Gibbs measure $\mu$ on $\Lambda$ (see e.g. \cite[\S{11}]{BDV}). It has positive entropy and full support since $\Lambda$ is $f$-minimal.

If the Lyapunov exponent of $\mu$ along the center direction is negative, then $\mu$-almost every point is the limit of periodic points $p_n$ having a uniformly large stable manifold $W^{cs}_{loc}(p_n)$ tangent to $E^s \oplus E^c$ (see \cite[\S 8]{ABC}). Since $p_n$ is close to $\Lambda$, for $n$ large enough the plaque $W^s_{loc}(p_n)$ intersects $\Lambda$. Since $\Lambda$ and the orbit of $p_n$ are invariant, this implies that the orbit of $p_n$ belongs to $\Lambda$. Since $\mu$ is fully supported, the set of such periodic points is dense in $\Lambda$.

If the Lyapunov exponent of $\mu$ along the center direction is non-negative and since $\mu$ has positive entropy,
the entropy of balls inside $\cW^s_{loc}(x)$ under negative iterations is positive for $\mu$-almost every $x$.
In particular $\cW^s_{loc}(x)\cap \Lambda$ has positive dimension for $x$ in a dense subset of $\Lambda$.
\end{proof}

\subsection{Unstable laminations with no periodic points}

\begin{prop}\label{p.transverse}
Let $\Lambda$ be an $f$-minimal unstable lamination which is strongly non-jointly integrable.
If there exists $z_0$ such that $\cW^s(z_0)\cap \Lambda$ contains at least $3$ points then $\Lambda$ is transverse.
\end{prop}

For the proof we modify Definition~\ref{def.transverse} and introduce another notion of transversality:
an $f$-invariant lamination $\Lambda$ is $\varepsilon$-$f$-transverse, if
there are $R>0$ and $N\geq 0$ such that for every unstable disk with internal radius $R$ in $\Lambda$
there exist $x\neq y$ in the disk and $0\leq n_x,n_y\leq N$ with the following property:
$x':=f^{n_x}(x)$ and $y':=f^{n_y}(y)$ belong to a same local stable manifold and the holonomy $\Pi^s_{x',y'}$ maps $\cW^u_\varepsilon(x')$ into both connected components of $D_y\setminus \cW^u_{loc}(y)$.
And $\Lambda$ is \emph{$f$-transverse} if it is $\varepsilon$-$f$-transverse for all $\varepsilon>0$.

Note that if $\Lambda$ is $f$-minimal and if there exist $x,y$ in a same local stable manifold
such that $\Pi^s_{x,y}\cW^u_\varepsilon(x)$ intersects both connected components of
$D_y\setminus \cW^u_{loc}(y)$, then $\Lambda$ is $\varepsilon$-$f$-transverse.
\medskip

We will first prove that $\Lambda$ is $f$-transverse and proceed by contradiction. We fix $\eps>0$ smaller than the constant $r_0$ of local product structure (see  \S\ref{ss.localproductstructure}). We consider $z_0\in\Lambda$ such that $\cW^s_{loc}(z_0)\cap \Lambda$ contains at least $3$ points and assume that $\Lambda$ is not $\eps$-$f$-transverse.

\begin{lema}\label{lem-openset}
There exists an open set $U\subset M$ with $U \cap \Lambda \neq \emptyset$ such that $\cW^u_{\eps}(y) \cap \cW^s_{loc}(z_0) \neq \emptyset$ for every $y \in \Lambda \cap U$.
\end{lema}

\begin{figure}[htbp]
\begin{center}
\includegraphics[scale=0.5]{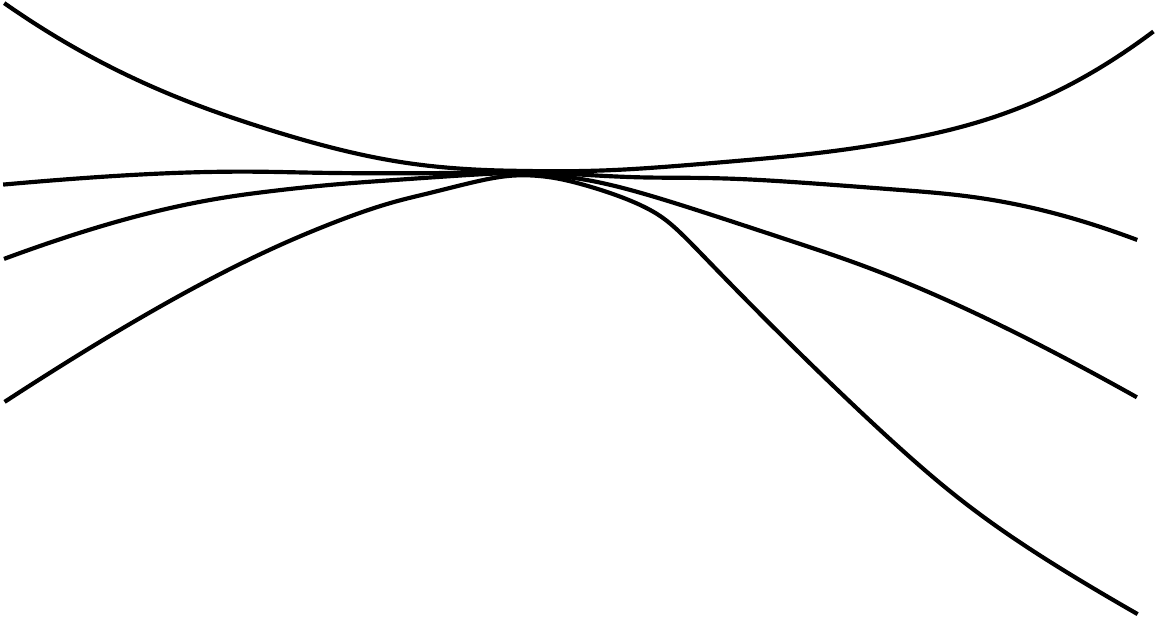}
\begin{picture}(0,0)
\end{picture}
\end{center}
\vspace{-0.4cm}
\caption{{\small When the lamination is not transverse but is strongly non-jointly integrable (picture after projection by $\Pi^s$-holonomy).}\label{f.nontransverse}}
\end{figure}

\begin{proof}
Fix some small $cu$-disk $D_{z_0}$ around $z_0$ which is separated by $\cW^u_{\eps}(z_0)$ and let us consider two points $z,z' \in \cW^s_{loc}(z_0)\setminus \{z_0\}$ close to $z_0$. The topological submanifolds $\Pi^s_{z,z_0}\cW^u_\varepsilon(z)$ and $\Pi^s_{z',z_0}\cW^u_\varepsilon(z')$ also separate $D_{z_0}$: this follows from the fact that the holonomies $\Pi^s_{z,z_0}$ and $\Pi^s_{z',z_0}$ are homeomorphisms from $cu$-disks $D_z$ and $D_{z'}$ to $D_{z_0}$.
Since $\Lambda$ is strongly non-jointly integrable these topological submanifolds do not coincide.

Considering a local transverse orientation, we can talk about being in the upper or lower component of $D_{z_0} \setminus \Pi^s_{y,z_0} \cW^{u}_{\eps}(y)$. Since $\Lambda$ is not $\varepsilon$-$f$-transverse, these submanifolds do not cross but intersect at $z_0$.

One of the three projected submanifolds separates the two other ones. Up to changing the roles of the points we can assume for instance that $\cW^{u}_{\eps}(z_0)$ is between $\Pi^s_{z',z_0}\cW^u_{\eps}(z')$ and $\Pi^s_{z,z_0}\cW^u_{\eps}(z)$ meaning that $\cW^u_{\eps}(z_0)$ is disjoint from the upper component of $D_{z_0} \setminus \Pi^s_{z',z_0} \cW^u_{\eps}(z')$ and from the lower component of $D_{z_0} \setminus \Pi^s_{z,z_0} \cW^u_{\eps}(z)$. 

Let $U$ be a small open neighborhood of $w \in \cW^u_\eps(z_0)$ which is disjoint from the projections $\Pi^s_{z',z_0}\cW^u_{\eps}(z')$ and $\Pi^s_{z,z_0}\cW^u_{\eps}(z)$. Fix $\beta \subset D_{z_0}$ an arc
through the point $z_0$ which intersects $\cW^u_{\eps}(z_0)$,  $\Pi^s_{z',z_0}\cW^u_{\eps}(z')$ and $\Pi^s_{z,z_0}\cW^u_{\eps}(z)$ only in $z_0$ and which is the union of two half arcs, one contained in the upper component of $D_{z_0} \setminus \Pi^s_{z',z_0} \cW^u_{\eps}(z')$ and one contained in  the lower component of $D_{z_0} \setminus \Pi^s_{z,z_0} \cW^u_{\eps}(z)$. For every $y\in U\cap \Lambda$ the plaque $\Pi^s_{y,z_0}\cW^u_{\eps}(y)$ is arbitrarily close to $\cW^u_{\eps}(z_0)$ in the Hausdorff topology, but does not cross $\Pi^s_{z,z_0}\cW^u_{\eps}(z)$ and $\Pi^s_{z',z_0}\cW^u_{\eps}(z')$. Applying \cite[Lemma 7.1]{CPoS} we deduce that $\Pi^s_{y,z_0} \cW^u_{\eps}(y)$ intersects $\beta$ at $z_0$ which is what we wanted to show. 
\end{proof}

We continue with the proof of Proposition~\ref{p.transverse}. We choose an open set $V$ which meets $\Lambda$ and whose closure is contained in $U$ given by Lemma \ref{lem-openset}.
As explained in \S~\ref{ss.localproductstructure}, one can assume that the metric is adapted and introduce:
$$1< \lambda_1 = \min \bigg\{ \frac{\|Df v\|}{\|v\|} \ : \ v \in E^u \bigg\} \leq \lambda_2 =  \max \bigg\{ \frac{\|Df v\|}{\|v\|} \ : \ v \in E^u \bigg\}.$$
Using the $f$-minimality of $\Lambda$, there is $n_0\geq 1$ such that for every $x \in \Lambda$, some iterate $f^{n}(\cW^u_{\lambda_2^{-1}\eps}(x))$ with $0\leq n\leq n_0$ contains a plaque $\cW^u_\eps(\bar x)$ with $\bar x\in V$.
In particular it meets $\cW^s_{loc}(z_0)$ at some point $x'$. Given $\delta>0$ small, by continuity of the local product structure there exists $\rho>0$ such that if $\bar y\in \Lambda$ satisfies $d(\bar x,\bar y)<\rho$ then $\cW^u_\eps(\bar y)$ also meets
$\cW^s_{loc}(z_0)$ and the intersection point $y'$ satisfies $d(x',y')<\delta$.
The number $\delta>0$ is chosen small enough so that if two points $x',y'\in \cW^s_{loc}(z_0)$ satisfy $d(x',y')<\delta$,
then $f^{-n}(y')\in \cW^s_{loc}(f^{-n}(x'))$.

By assumption, there are points $x_0, y_0 \in \Lambda$ that belong to the same local stable manifold. Up to take some iterates, we can assume that the $C^0$-distance between $\cW^{u}_{loc}(x_0)$ and $\cW^{u}_{loc}(y_0)$ is smaller than $\rho$.
We will define inductively points $x_k,y_k\in\Lambda$ which belong to the same local stable manifold and such that
$d(x_k,f(x_{k-1}))<\lambda_2\eps$ and $d(y_k,f(y_{k-1}))<\lambda_2\eps$. Since the stable distance is contracted by forward iterations, these properties imply that the $C^0$-distance between $\cW^{u}_{loc}(x_k)$ and $\cW^{u}_{loc}(y_k)$ remains smaller than $\rho$.

For each $k$ we define: 
 $$ \Omega_k = \{z \in \cW^{u}_{loc}(x_k) \ : \ \cW^{s}_{loc}(z) \cap \cW^u_{loc}(y_k) = \emptyset \}. $$
By strong non-joint integrability, this is an open and dense subset of $ \cW^{u}_{loc}(x_k)$.
Since $\cW^u_{loc}(x_k)\subset f(\cW^u_{loc}(x_{k-1}))$, $\cW^u_{loc}(y_k)\subset f(\cW^u_{loc}(y_{k-1}))$,
and since the distance between $x_k,y_k$ remains small,
we have $\Omega_{k}=f(\Omega_{k-1})\cap \cW^u_{loc}(x_k)$. Note also that by definition $x_k\not\in \Omega_k$.

Assuming that $x_{k}$ has been defined, the construction of $x_{k+1},y_{k+1}$ is done as follows.
Let us consider a ball $B_{k}$ of maximal radius in $\Omega_{k} \cap \cW^u_{\eps}(x_{k})$:
the point $x_{k+1}$ is chosen so that $f^{-1}(x_{k+1})$ belongs to the boundaries of $\Omega_{k}$ and of $B_{k}$.
While the radius $a_{k}$ of $B_{k}$ satisfies $\lambda_2 a_{k} < \eps$ we have
$$\lambda_1 a_{k}<a_{k+1}<\lambda_2 a_{k}.$$
Eventually we reach the size $ \lambda^{-1}_2 \eps$ at some step $\ell$ and stop the construction.

By our choices above, there exists $n\leq n_0$ such that $f^n(B_\ell))$ meets $\cW^s_{loc}(z_0)$ at some point $x'$.
Since $y_\ell$ is $\rho$-close to $x_\ell$, the image $f^n(\cW^s_{loc}(y_\ell))$ also meets $\cW^s_{loc}(z_0)$ at some point $y'$
and $d(x',y')<\delta$. This implies that $\Omega_\ell$ contains a point $f^{-n}(x')$ such that
$\cW^s_{loc}(f^{-n}(x'))$ meets $\cW^u_{loc}(y_{\ell})$, which contradicts the definition of $\Omega_\ell$.

We have thus obtained a contradiction and proved that $\Lambda$ is $\varepsilon$-$f$-transverse.
Since $\varepsilon$ is arbitrary, $\Lambda$ is $f$-transverse. One concludes the proof of the proposition with the following lemma.

\begin{lema}
If $\Lambda$ is an $f$-minimal $f$-transverse lamination, then it is the finite union of minimal laminations and it is transverse.
\end{lema}
\begin{proof}

Let $\Lambda'\subset \Lambda$ be a minimal lamination (it is not invariant if $\Lambda$ is not minimal).
Let us fix $\varepsilon>0$ small so that the $\varepsilon$-$f$-transversality is satisfied.
In particular there exist $x,y\in \Lambda'$ and $0\leq n_x,n_y$ with $f^{n_x}(x)\in \cW^s_{loc}(f^{n_y}(y))$
such that the holonomy $\Pi^s_{x',y'}$ maps $\cW^u_\varepsilon(x')$ into both connected components of $D_y\setminus \cW^u_{loc}(y)$, where $x'=f^{-n_x}(x)$ and $y'=f^{-n_y}(y)$. By $f$-minimality, there exists $N\geq 0$ and
for any $m\geq 0$, there exist $0\leq j,k \leq N$ such that $f^{-m-j}(\Lambda')$ and $f^{-m-k}(\Lambda')$
contain discs respectively close to $f^{-n_x}(\Lambda')$ and $f^{-n_y}(\Lambda')$.
In particular, there exists a sequence $m_k\to\infty$ and $\ell\geq 0$ such that
$f^{-m_k}(\Lambda)$ and $f^{-m_k-\ell}(\Lambda')$ share two points in a same local stable manifold.
Consequently $\Lambda'=f^\ell(\Lambda')$. One deduces that $\Lambda$ is the union of the $\ell$ first iterates of $\Lambda'$.
Choosing $\ell$ minimal, this union is disjoint.

Note that the stable leaves of $\lambda'$ cannot intersect $f^{i}(\Lambda')$ for $1\leq i<\ell$ since otherwise $\Lambda'=f^i(\Lambda')$ contradicting the minimality of $\ell$. The $\eps$-$f$-transversality of $\Lambda$ is thus satisfied with $n_x=n_y$ for any unstable disc. The minimality of each iterate $f^i(\Lambda')$ and the periodic decomposition imply that $n_x,n_y$ can be chosen in $[0,\ell]$. Hence the $\eps$-$f$ transversality is satisfied with $N=0$, provided $R>0$ is chosen large enough. This implies that $\Lambda$ is transverse.
\end{proof}

\subsection{Decomposition of strongly non-jointly integrable laminations}\label{s-njifinite}
We prove here Remark~\ref{r-njifinite}. Let $\Lambda$ be an $f$-minimal lamination which is strongly non-jointly integrable.
Let $\Lambda'\subset \Lambda$ be a minimal lamination. Its backward orbit is dense inside $\Lambda$.
We apply the dichotomy given by Proposition~\ref{prop-lyapunovexponents}.

If $\Lambda$ contains a periodic orbit with negative center Lyapunov exponent, then its local stable manifold intersects large backward iterates of $\Lambda'$. This implies that $\Lambda'$ intersects arbitrarily small neighborhood of the periodic orbit.
Hence the orbit of $\Lambda'$ is finite and $\Lambda$ is the union of these iterates.

If there exists $z_0$ such that $\cW^s_{loc}(z_0)\cap \Lambda$ contains at least three points, then by Proposition~\ref{p.transverse} $\Lambda$ is transverse. We can thus apply Proposition~\ref{prop-njifinite} and conclude that $\Lambda$ is a finite union of minimal laminations.

\subsection{Generic consequence, proof of Corollary \ref{C.QAorTrans}} 
Let us consider a dense G$_\delta$ subset $\cG \subset \PH^1_{c=1}(M)$ such that Proposition \ref{prop-QAgeneric}, Theorem~\ref{t.palis} and Theorem \ref{T.RobustSH} hold.
Let $f \in \cG$ and $\Lambda$ be an $f$-minimal unstable lamination.
Theorem \ref{T.RobustSH} implies that $\Lambda$ is strongly non-jointly integrable.
By Theorem \ref{t.transverse-periodic} if the unstable lamination $\Lambda$ does not contain any periodic point, then it is transverse and the conclusion of Corollary \ref{C.QAorTrans} holds.

We are thus reduced to the case where $\Lambda$ contains a periodic point and by Proposition~\ref{prop-QAgeneric} it is a quasi-attractor. If $\Lambda=M$, then the foliation $\cW^u$ is $f$-minimal. By Proposition~\ref{p.transverse0} the unstable foliation is transverse and the conclusion of Corollary \ref{C.QAorTrans} is satisfied. We can thus now assume that $\Lambda$ is a proper quasi-attractor. By Remark~\ref{r-njifinite}, $\Lambda$ is a union of finitely many disjoint minimal unstable laminations; up to replace $f$ by an iterate, we can thus reduce to the case where $\Lambda$ is a minimal unstable lamination.

We first address the case where $\Lambda$ contains a periodic point $q$ whose center direction is expanded.
Any small disc $D^{cu}_q$ centered at $q$ inside the center-unstable manifold of $q$ is contained in $\Lambda$.
Let us consider some point $z\in \cW^s_{loc}(q)\cap \Lambda$ and let $\varepsilon>0$. By the strong non-joint integrability, $\Pi^s_{q,z}(\cW^u_{\varepsilon}(z))$ is not included inside $\cW^u(q)$, hence it crosses both connected components of $D^{cu}_q\setminus \cW^u_{loc}(y)$ for some $y\in D^{cu}_q$ close to $q$.
By minimality of the lamination $\Lambda$, any unstable leaf inside $\Lambda$ contains subdiscs close to $\cW^u_{\varepsilon}(z)$ and $\cW^u_{loc}(y)$. This shows that $\Lambda$ is $\varepsilon$-transverse. We deduce that $\Lambda$ is transverse and the conclusion of Corollary \ref{C.QAorTrans} is satisfied.
\smallskip

We conclude the proof by considering three cases:

If $\Lambda$ is uniformly hyperbolic and its center bundle is uniformly contracted, it is a proper transitive attractor, and the conclusion of Corollary \ref{C.QAorTrans} holds.

If $\Lambda$ is uniformly hyperbolic and its center bundle is uniformly expanded, $\Lambda$ contains a periodic point whose center direction is uniformly expanded and we conclude by the argument above.

If $\Lambda$ is not uniformly hyperbolic, by Theorem~\ref{t.palis} it contains a periodic point whose center direction is expanded and we conclude as before.
\smallskip

The argument works identically inside $\PH^1_{c=1,\mathrm{vol}}(M)$ because those results are satisfied in that setting too.

\section{Section for the dynamics}\label{s.sectiondynamics}

Here we show Theorem \ref{T.SHorQA}. The proof first reduces to the case where the set is not a quasi-attractor, and uses Proposition \ref{prop-lamtransvcantor} to conclude that $\Lambda$ is transversally Cantor. This will allow us to construct a section for the dynamics and then to apply Theorem \ref{teo-hyp} to obtain the SH property along the minimal unstable lamination $\Lambda$.

\subsection{A section for the dynamics}

We first show the following property:

\begin{prop}\label{p.section}
Let $\Lambda$ be an $f$-minimal unstable lamination which is transversally Cantor. Then, there exists $N\geq 1$ and a compact $f^N$-invariant subset $K \subset \Lambda$ such that $\cW^u(x) \cap K = \{x\}$ for every $x\in K$ and $\cW^u(y) \cap W^s(K) \neq \emptyset$ for every $y \in \Lambda$.  
\end{prop}

We emphasize that $W^s(K)$ here does not denote the union of the strong stable manifolds through points of $K$ but the stable set:
$$W^s(K) = \{ z \in M \ : \ d(f^{kN}(z), f^{kN}K) \to 0 \ \text{ as } k \to +\infty\}.$$
In particular $\cW^u(y)\cap W^s(K) \neq \emptyset$ means that for every $\eps>0$ there is $k_0$ (which does not depend on $y$) and $z \in \cW^u_{loc}(y)$ such that $d(f^{kN}(z), K) < \eps$ for every $k>k_0$. 

\begin{proof}
Consider a submanifold with boundary $\Sigma$ transverse to $E^u$ which is a global section of $\Lambda$ and such that $\partial \Sigma \cap \Lambda = \emptyset$: it exists since $\Lambda$ is totally disconnected along the center-stable direction. Choose, for each $x \in \Lambda \cap \Sigma$ a disk $D_x \subset \cW^u_{loc}(x)$ centered at $x$ varying continuously with $x$  and such that $D_x\cap \Sigma=\{x\}$. Since $f$ expands $\cW^u$ leaves, there is some $N>0$ such that for every $x \in \Lambda \cap \Sigma$ the image $f^N(D_x)$ contains some $D_y$, $y \in \Lambda \cap \Sigma$. 
Since $\Lambda \cap \Sigma$ is totally disconnected and since the discs $D_x$ are pairwise disjoint, there exists a continuous injective function $\phi: \Lambda \cap \Sigma \to \Lambda \cap \Sigma$,  such that:  
$$ f^{N} (D_x) \supset D_{\phi(x)} . $$
 
Let $\cA = \bigcap_{k > 0} \phi^k (\Lambda \cap \Sigma)$: it is compact and non-empty. The restriction $\phi|_\cA$ is a homeomorphism of $\cA$: we have to check that $\phi\colon \cA\to \cA$ is surjective; as the intersection $\bigcap_{k > 0} \phi^k (\Lambda \cap \Sigma)$ is decreasing, it coincides with $\bigcap_{k > 1} \phi^k (\Lambda \cap \Sigma)$, proving that $\cA=\phi(\cA)$.

Now, let us consider $\cB = \bigcup_{x \in \cA} D_{\phi(x)}$. It is backward $f^N$-invariant since
for each $x\in \cA$ there is $y \in \cA$ such that $\phi(y)=x$ and since $f^{-N}(D_{\phi(y)}) \subset D_y$.  So, we can introduce the compact $f^N$-invariant set $ K = \bigcap_{k>0} f^{-kN}(\cB)$.

Let us see that $K$ has the desired properties. If $y \in \Lambda$, its unstable manifold intersects $\Sigma \cap \Lambda$ and contains a disc $D_x$. Therefore $\cW^u(y)$ contains a point $z$ whose forward iterates under $f^N$ belong to $\cB$.
For $\ell\geq 1$,, the iterate $f^{\ell N}(z)$ belongs to $\bigcap_{\ell\geq k>0} f^{-kN}(\cB)$ which is contained in an arbitrarily small neighborhood of $K$ when $\ell$ is large enough. Hence $z$ belongs to $W^s(K)$.

Finally, let us consider two points $z,z'\in K$ which belong to a same unstable manifold.
Up to replacing them by backward iterates, we may assume that they belong to a same disc $D_x$, $x\in \cA$.
By definition of $K$ and injectivity of $\phi$, their forward iterates $f^{kN}(z)$ and $f^{kN}(z')$ belong to a same disc $D_{x_k}$.
Since unstable leaves are uniformly expanded, we conclude that $z=z'$, proving that each unstable leaf intersects $K$ in at most one point.\end{proof}

\subsection{SH property, proof of Theorem \ref{T.SHorQA}} 
We explain the proof in the case of $\PH^1_{c=1}(M)$ as the proof in the $\PH^1_{c=1,vol}(M)$ is identical given that we have already proved the used results in both settings. 
We consider a dense G$_\delta$ subset $\cG\subset \PH^1_{c=1}(M)$ such that any iterate of diffeomorphisms $f\in \cG$ satisfies Proposition \ref{prop-QAgeneric}, Corollary~\ref{c.palis} and Theorems \ref{teo-hyp} and \ref{T.RobustSH}. Let $f \in \cG$ and let $\Lambda$ be a $f$-minimal unstable lamination. By Theorem~\ref{T.RobustSH}, $\Lambda$ is strongly non-jointly integrable.
We consider different cases.
\smallskip

If $\Lambda$ is uniformly hyperbolic, either $E^c|_\Lambda$ is uniformly expanded and the property SH holds on $\Lambda$, or $E^c|_\Lambda$ is uniformly contracted, $\Lambda$ is a uniformly hyperbolic set saturated by its unstable manifolds and is an attractor. This concludes the proof of the theorem in this case.

If $\Lambda$ is a quasi-attractor and not uniformly hyperbolic, by Corollary~\ref{c.palis} and the $f$-minimality of $\Lambda$, there exists a uniformly hyperbolic set $K\subset \Lambda$ such that $E^c|_K$ is uniformly expanded and there exists $N\geq 1$ such that for any $y\in \Lambda$, the unstable leaf $f^N\cW^u_{loc}(y)$ intersects some $\cW^s_{loc}(x)$, $x\in K$ at some point $z$.
The center space $E^c(z)$ is uniformly expanded under forward iterations. This proves that Property SH holds on $\Lambda$.

If $\Lambda$ contains a periodic point and is not uniformly hyperbolic, Proposition~\ref{prop-QAgeneric} shows that it is a quasi-attractor and the property SH holds as in the previous case.

If $\Lambda$ is not transversally Cantor and is not uniformly hyperbolic, by Proposition~\ref{prop-lamtransvcantor} there exists a non-empty open set of points $x$ satisfying $\omega(x)\subset \Lambda$. By Proposition~\ref{prop-QAgeneric} the set $\Lambda$ is a quasi-attractor and satisfies the property SH as explained before.
\smallskip

It remains to consider the case where $\Lambda$ is not uniformly hyperbolic, does not contain a periodic point, and is transversally Cantor. We first apply Proposition~\ref{p.section}: there is an $f^N$-invariant set $K \subset \Lambda$ such that $\cW^u(y) \cap W^s(K) \neq \emptyset$ for every $y \in \Lambda$. Moreover every unstable leaf intersects $K$ in at most one point, so that, by Theorem \ref{teo-hyp}, the set $K$ is hyperbolic for $f^N$. Note that the center direction is not uniformly contracting: by the shadowing lemma, there would exist a periodic orbit contained in a small neighborhood of $K$ so that its stable manifold would intersect a $\cW^u$-leaf of $\Lambda$; this would imply that $\Lambda$ contains a periodic point, 
and contradict our assumption. So the bundle $E^c$ is uniformly expanding on $K$. 

By Proposition~\ref{p.section}, there exists $k_0\geq 0$ such that for every $y\in \Lambda$ there exists
$z\in \cW^u_{loc}(y)$ whose forward iterates $f^{kN}(z)$, $k\geq k_0$, are all contained in a small neighborhood of $K$; in particular the center bundle at $z$ is uniformly expanded under forward iterations.
This shows that $\Lambda$ verifies the SH property and completes the proof of the theorem.


\bigskip

\noindent
\emph{Sylvain Crovisier}\\
{CNRS -- Laboratoire de Math\'ematiques d'Orsay (UMR 8628),\\
Universit\'e Paris-Saclay, 91405 Orsay, France.}\\
\url{www.imo.universite-paris-saclay.fr/~sylvain.crovisier/}\\
\texttt{sylvain.crovisier@universite-paris-saclay.fr}

\bigskip

\noindent
\emph{Rafael Potrie}\\
{CMAT, Facultad de Ciencias, Universidad de la Rep\'ublica, Uruguay, and},\\
{CNRS -- Laboratorio del Plata (IRL IFUMI-2030)}\\
\url{www.cmat.edu.uy/~rpotrie/}\\
\texttt{rpotrie@cmat.edu.uy}
\bigskip


\begin{thebibliography}{2}


\bibitem[ABC]{ABC} F. Abdenur, C. Bonatti, S. Crovisier,  Non-uniform hyperbolicity for $C^1$-generic diffeomorphisms. \emph{Israel J. of Math.} {\bf 183} (2011), 1--60.

\bibitem[AC]{AC} F. Abdenur, S. Crovisier,  Transitivity and topological mixing for $C^1$ diffeomorphisms.
In \emph{Essays in Mathematics and its Applications: In Honor of Stephen Smale's 80th Birthday.} Springer (2012), 1--16.

\bibitem[ALOS]{ALOS} S. Alvarez, M. Leguil, D. Obata, B. Santiago, Rigidity of u-Gibbs measures near conservative Anosov diffeomorphisms on $\mathbb{T}^3$. \emph{Journal of the EMS} \textbf{28} (2026), 2659--2750.

\bibitem[A]{Avila} A. Avila, On the regularization of conservative maps. \emph{Acta Math.}  \textbf{205} (2010), 5--18. 

\bibitem[ACEPWZ]{ACEPWZ} A. Avila, S. Crovisier, A. Eskin, R. Potrie, A. Wilkinson, Z. Zhang.
\emph{Unstable foliation dynamics in dimension 3}. In preparation.

\bibitem[ACW$_1$]{ACW-blender} A. Avila, S. Crovisier, A. Wilkinson, $C^1$-density of stable ergodicity. \emph{Adv. Math.} {\bf 379} (2021), paper No. 107496, 68 pp.

\bibitem[ACW$_2$]{ACW} A. Avila, S. Crovisier, A. Wilkinson, Minimality of strong foliations of Anosov and partially hyperbolic diffeomorphisms. ArXiv:2504.01085.


\bibitem[BC$_1$]{BC-rec} C. Bonatti, S. Crovisier, R\'ecurrence et g\'en\'ericit\'e. \emph{Inventiones Math.} {\bf 158} (2004), 33--104.

\bibitem[BC$_2$]{BC-Center} C. Bonatti, S. Crovisier, Center manifolds for partially hyperbolic set without strong unstable connections. \emph{Journal of the IMJ} {\bf15} (2016), 785--828.

\bibitem[BCGP]{BCGP} C. Bonatti, S. Crovisier, N. Gourmelon, R. Potrie,  Tame dynamics and robust transitivity: chain-recurrence classes versus homoclinic classes. \emph{Transactions of the AMS} {\bf 366} (2014), 4849--4871.

\bibitem[BDP]{BDP} C. Bonatti, L. Diaz, E. Pujals, A $C^1$-generic dichotomy for diffeomorphisms: weak forms of hyperbolicity or infinitely many sinks or sources. \emph{Ann. of Math.} {\bf 158} (2003), 355--418.


\bibitem[BDU]{BDU} C. Bonatti, L. Diaz, R. Ures, Minimality of strong stable and unstable foliations for partially hyperbolic diffeomorphisms. \emph{J. Inst. Math. Jussieu} {\bf 1} (2002), 513--541.

\bibitem[BDV]{BDV} C. Bonatti, L. D\'iaz and M. Viana, \emph{Dynamics beyond uniform hyperbolicity. A global geometric and probabilistic perspective}. Encyclopaedia of Mathematical Sciences {\bf 102}. Mathematical Physics III. Springer-Verlag (2005).

\bibitem[BGHP]{BGHP} C. Bonatti, A. Gogolev, A. Hammerlindl, R. Potrie,  Anomalous partially hyperbolic diffeomorphisms III: Abundance and incoherence. \emph{Geom. Topol.} {\bf 24} (2020), 1751--1790.

\bibitem[BMVW]{BMVW} C. Bonatti, C. Matheus, M. Viana, A. Wilkinson,  Abundance of stable ergodicity. \emph{Comment. Math. Helv.} {\bf 79} (2004), 753--757.


\bibitem[C]{Crov-habilitation} S. Crovisier, Perturbation de la dynamique de diff\'eomorphismes en topologie $C^1$. \emph{Ast\'erisque} \textbf{354} (2013). 


\bibitem[CPo$_1$]{CroPot} S. Crovisier, R. Potrie, Introduction to partially hyperbolic dynamics,  \emph{Lecture notes for School on Dynamical Systems}, ICTP, Trieste (2015). \emph{Unpublished.}

\bibitem[CPo$_2$]{CPo} S. Crovisier, R. Potrie, Attractors for derived from Anosov diffeomorphisms. \emph{Manuscript.}

\bibitem[CPoS]{CPoS} S. Crovisier, R. Potrie, M. Sambarino, Finiteness of partially hyperbolic attractors with one-dimensional center.  \emph{Ann. Sci. \'Ec. Norm. Sup\'er.} {\bf 53} (2020), 559--588.

\bibitem[CPu]{CP} S. Crovisier, E. Pujals, Essential hyperbolicity and homoclinic bifurcations: a dichotomy phenomenon/mechanism for diffeomorphisms. \emph{Inventiones Mathematicae} {\bf 201} (2015), 385--517. 

\bibitem[CPuS]{CPS} S. Crovisier, E. Pujals, M. Sambarino, Hyperbolicity of the extreme bundles. \emph{In preparation}.

\bibitem[EPZ]{EPZ} A. Eskin, R. Potrie, Z. Zhang, Geometric properties of partially hyperbolic measures and applications to measure rigidity. ArXiv:2302.12981.

\bibitem[HT]{HT} A. Herrera, A. Tercia, Robust transitivity and density of periodic points of partially hyperbolic diffeomorphisms. ArXiv:1403.3979.

\bibitem[HHU]{HHU} J. Rodriguez Hertz, F. Rodriguez Hertz, R. Ures,  Some results on the integrability of the center bundle for partially hyperbolic diffeomorphisms. \emph{Fields Inst. Commun.} \textbf{51}.
American Mathematical Society, Providence, RI (2007) 103--109.


\bibitem[HUY]{HUY} J. Rodriguez Hertz, R. Ures, J. Yang, Robust minimality of strong foliations for DA diffeomorphisms: cu-volume expansion and new examples. \emph{Trans. Amer. Math. Soc.} \textbf{375} (2022), 4333--4367.

\bibitem[HPS]{HPS} M. Hirsch, C. Pugh and M. Shub, Invariant Manifolds.  \emph{Lecture Notes in Math.} {\bf 583} (1977).

\bibitem[KH]{KH} A. Katok, B. Hasselblatt, Introduction to the modern theory of dynamical systems.
\emph{Encyclopedia Math. Appl.} \textbf{54}
Cambridge University Press, Cambridge, (1995).

\bibitem[Ka]{Katz} A. Katz, Measure rigidity of Anosov flows via the factorization method. \emph{Geom. Funct. Anal.} {\bf 33} (2023), 468--540.

\bibitem[PuSa$_1$]{PS-SH} E. Pujals, M. Sambarino, A sufficient condition for robustly minimal foliations. \emph{Ergodic Theory Dynam. Systems} {\bf 26} (2006), 281--289.

\bibitem[PuSa$_2$]{PS-IHP} E. Pujals, M. Sambarino, Density of hyperbolicity and tangencies in sectional dissipative regions. \emph{Ann. Inst. H. Poincar\'e C Anal. Non Lin\'eaire} {\bf 26} (2009), 1971--2000.



\bibitem[Ya]{JYang} J. Yang, Entropy along expanding foliations. \emph{Advances in Math.}  {\bf 389} (2021) 107893.

\end{thebibliography}
\end{document}